\documentclass[11pt,a4paper,draft]{article}

\usepackage[ngerman,english]{babel}
\usepackage[utf8]{inputenc}

\usepackage[DIV=12]{typearea}

\usepackage[affil-it]{authblk}

\usepackage{xparse}

\usepackage{amsmath, amsfonts, amssymb, amsthm}
\usepackage{mathrsfs}
\usepackage{mathtools}

\usepackage{enumerate}
\usepackage{color}

\usepackage[noadjust]{cite}

\usepackage[capitalise]{cleveref}

\usepackage{ragged2e}
\usepackage{url}

\usepackage{xstring}

\crefname{equation}{}{}

\newcommand{\R}{\mathbb{R}}
\newcommand{\Z}{\mathbb{Z}}
\newcommand{\N}{\mathbb{N}}

\newcommand{\ee}{\mathrm{e}}

\DeclareDocumentCommand\dd{ o g d() }{
	\IfNoValueTF{#2}{
		\IfNoValueTF{#3}
			{\mathrm{d}\IfNoValueTF{#1}{}{^{#1}}}
			{\mathinner{\mathrm{d}\IfNoValueTF{#1}{}{^{#1}}\argopen(#3\argclose)}}
		}
		{\mathinner{\mathrm{d}\IfNoValueTF{#1}{}{^{#1}}#2} \IfNoValueTF{#3}{}{(#3)}}
	}

\newcommand{\dx}{\dd{x}}
\newcommand{\dy}{\dd{y}}

\newcommand{\ds}{\dd{s}}
\newcommand{\dt}{\dd{t}}

\newcommand{\del}{\partial}
\newcommand{\eps}{\varepsilon}
\newcommand{\im}{\mathrm{i}}

\newcommand{\K}{\mathscr{K}} 
\newcommand{\SW}{\mathscr{S}}
\newcommand{\B}{\mathscr{P}}

\DeclareMathOperator{\dist}{dist}
\DeclareMathOperator{\pernorm}{d}

\let\P\relax 
\DeclareMathOperator{\P}{\mathbb{P}} 

\DeclareMathOperator{\E}{\mathbb{E}} 

\newcommand{\vcc}{\vcentcolon}

\DeclarePairedDelimiter\abs{\lvert}{\rvert}
\DeclarePairedDelimiter\norm{\Vert}{\rVert}

\theoremstyle{plain}
\newtheorem{theorem}{Theorem}[section]
\newtheorem{lemma}[theorem]{Lemma}
\newtheorem{proposition}[theorem]{Proposition}

\theoremstyle{definition}

\theoremstyle{remark}
\newtheorem{remark}[theorem]{Remark}

\title{Power Network Dynamics on Graphons}

\date{}

\begin{document}

 \author{Christian Kuehn\thanks{Technical University of Munich, Faculty of Mathematics, Research Unit 
``Multiscale and Stochastic Dynamics'', 85748 Garching b.~M\"unchen, Germany, \texttt{ckuehn@ma.tum.de}}~
and~Sebastian Throm\thanks{Technical University of Munich, Faculty of Mathematics, Research Unit 
``Multiscale and Stochastic Dynamics'', 85748 Garching b.~M\"unchen, Germany, \texttt{throm@ma.tum.de}} }

\maketitle

\begin{abstract}
Power grids are undergoing major changes from a few large producers to smart
grids build upon renewable energies. Mathematical models for power grid dynamics
have to be adapted to capture, when dynamic nodes can achieve synchronization to a common
grid frequency on complex network topologies. In this paper we study a second-order 
rotator model in the large network limit. We merge the recent theory of random graph 
limits for complex small-world networks with approaches to first-order systems on 
graphons. We prove that there exists a well-posed continuum limit integral equation 
approximating the large finite-dimensional case power grid network dynamics. Then we
analyse the linear stability of synchronized solutions and prove linear stability. However, 
on small-world networks we demonstrate that there are topological parameters moving the
spectrum arbitrarily close to the imaginary axis leading to potential instability on
finite time scales.  
\end{abstract}

\section{Introduction}\label{Sec:Introduction}

In this work we focus on the study of the system of differential equations
\begin{equation}
\label{eq:Braess:0}
 \frac{\dd^2}{\dt^2}\phi_{k}=-\alpha \frac{\dd}{\dt}\phi_{k}+
\sum_{\ell=1}^{N}K_{k,\ell}\sin(\phi_{\ell}-\phi_{k})+P_{k},\qquad \phi_k=\phi_k(t).
\end{equation}
where $\phi_k$ is the phase of the $k$-th component of an electro-mechanical element 
in a power grid/network with $k\in\{1,2,\ldots,N\}$, $\alpha\in\R$ is a damping 
parameter, $P_k\in\R$ represents production ($P_k>0$) or consumption ($P_k<0$), and
$K_{k,l}$ is the coupling matrix between elements. In this introduction, we provide a 
non-technical overview of our results. The technical part of this work starts 
in Section~\ref{Sec:derivemodel} with the derivation of the model~\eqref{eq:Braess:0}.
The basic idea of \eqref{eq:Braess:0} is that $\phi_k$ represents a phase difference to
a standard reference phase $\mathfrak{F}$, to which all producers and consumers in a 
power grid~\cite{MachowskiBialekBumby} should synchronize. In particular, 
the model~\eqref{eq:Braess:0}
describes self-synchronization~\cite{Kur84,PikovskyRosenblumKurths,Strogatz2} 
effects without external controls. The ODEs~\eqref{eq:Braess:0}, and several of its 
variations have received very significant attention 
recently in power grid 
modelling~\cite{MenckHeitzigKurthsSchellnhuber,Motteretal,NishikawaMotter,
RohdenSorgeTimmeWitthaut,RohdenSorgeWitthautTimme,WiT12}; see also~\cite{DoerflerBullo} 
for a broader survey. A key motivation for this surging interest in self-synchronization
of power grids arises due to the paradigmatic shift from very few central power
plant sources to a multi-faceted network of smaller producers via renewable 
energies (solar, wind, etc). For a distributed grid, it is far more challenging 
to properly judge the influence of controls, showing the pressing need to design
``smart grids''. In this context, a natural first question to study
is, under which conditions a power grid self-synchronizes without control. 

For low-dimensional variants of~\eqref{eq:Braess:0}, i.e., for very small $N$, 
classical non-linear dynamics and bifurcation theory provide already many
insightful answers if one studies the existence and stability of steady states 
for~\eqref{eq:Braess:0}. However, even in this case, potentially counter-intuitive
effects can appear. For example, the addition of a line between oscillators 
corresponding to changing $K_{k,\ell}$ from a zero to a non-zero value can 
destabilize synchronous solutions~\cite{WiT12,ColettaJacquod}. This effect is a
manifestation of Braess' paradox~\cite{Braess} well-known from transportation
networks~\cite{SteinbergZangwill}. However, small grids do not properly 
represent the large-scale modern distributed grids, so we have to 
study~\eqref{eq:Braess:0} for large $N$. Evidently one then has to ask about
the topology of the graph $G_{N}$ with (weighted) adjacency 
matrix $(K_{k,\ell})_{k,\ell=1}^N$. Modern power grids have a complex network
structure~\cite{AlbertBarabasi,New03} well-characterized by the many available
large-scale power-law distribution random graph models, such as small-world
networks~\cite{WattsStrogatz}. Although there are many excellent 
numerical and formal studies of~\eqref{eq:Braess:0} for complex 
networks~\cite{MenckHeitzigKurthsSchellnhuber,Motteretal,NishikawaMotter,
RohdenSorgeTimmeWitthaut,RohdenSorgeWitthautTimme,WiT12}, there are currently 
not many available rigorous analytical methods/studies.

In this work, we provide a novel tool to rigorously approximate the dynamics
of complex network power grid models of the form~\eqref{eq:Braess:0} via a continuum 
limit as $N\rightarrow \infty$. In our context, we show how to transfer the recently 
emerging theory of integral equations on graphons~\cite{Lovasz,LovaszSzegedy} to 
power network models, and more general second-order ODE systems. Graphons are 
natural limit objects for certain classes of graphs $G_N$ as $N\rightarrow \infty$. 
Graphons can be viewed as functions over the unit square 
$I\times I=[0,1]\times [0,1]$. They are obtained by using the adjacency matrix of the 
graph to generate a step function $\K^{(N)}(x,y)$ over $I\times I$. Upon a suitable 
rescaling and using a reduction by graph homomorphisms, the step function converges
as $N\rightarrow \infty$ to the graphon $\K^{(\infty)}(x,y)$; see also 
Section~\ref{Sec:model}. It has
been shown in an important recent sequence of works~\cite{Med14,Med14a,KaM17} that certain 
classes of first-order ODEs on complex networks $K_{k,\ell}$ can lead to a continuum limit, 
which is a non-local analogue of classical local reaction-diffusion partial differential
equations (PDEs). In particular, if we start with second-order ODEs of the form
\begin{equation}
\label{eq:ODEintro}
 \frac{\dd^{2}}{\dd{t}^2}\phi_{k}^{(N)}=-\alpha\frac{\dd}{\dd{t}}\phi_{k}^{(N)}
+ \frac{1}{N}\sum_{\ell=1}^{N}K_{k,\ell}^{(N)}D(\phi_{\ell}^{(N)}-\phi_{k}^{(N)})
+f(\phi_{k}^{(N)},t),
\end{equation}
for $k\in\{1,2,\ldots,N\}$, then the continuum limit turns out to be
\begin{equation}
\label{eq:ODE:as:inteqintro}
 \frac{\dd^{2}}{\dd{t}^2}\phi(x,t)+\alpha\frac{\dd}{\dd{t}}\phi(x,t)=
\int_{I}\K^{(\infty)}(x,y)D\bigl(\phi(y,t)-\phi(x,t)\bigr)\dy+f\bigl(\phi(x,t),t\bigr)
\end{equation}
under suitable conditions upon the functions $D$ and $f$. To incorporate the idea 
of complex power grid topologies, we are going to establish the limiting equation
also in the context, when the graph $G_N$ is random.
Using the continuum limit 
of~\eqref{eq:Braess:0}, we are going to study stability of homogeneous steady 
states for networks of small-world type. This provides insight into the effect 
of the \emph{graph topology on stability}. For example, we show
that decreasing the density of long-range connections decreases a stability index. 
In summary our main results in a non-technical form are the following:

\begin{itemize}
 \item[(R1)] Under suitable conditions, the models \eqref{eq:ODEintro} 
and~\eqref{eq:ODE:as:inteqintro} are well-posed (Proposition~\eqref{Prop:well:posed:cont:limit}). 
 \item[(R2)] On any finite time interval $[0,T]$, the second-order 
ODEs~\eqref{eq:ODEintro} converge in $L^2(I\times I)$ to the continuum 
limit~\eqref{eq:ODE:as:inteqintro} if the initial conditions for both models converge
to each other and graphon convergence takes places 
as $N\rightarrow \infty$ (Theorem \ref{Thm:continuum:limit}).
 \item[(R3)] The same result as in (R2) applies also to certain classes of random
graphs in a suitable averaged sense (Theorem \ref{Thm:continuum:limit:random}).
 \item[(R4)] For $D(\cdot)=\sin(\cdot)$, a balanced power grid, and for a class of 
symmetric graphon kernels $\K^{(\infty)}(x,y)=\K^{(\infty)}(x-y)$, there exist families 
of homogeneous steady states to \eqref{eq:Braess:0}, which are linearly stable
(Proposition~\ref{Prop:linear:stability}) for the physical situation of positive 
damping with $\alpha<0$.
 \item[(R5)] Yet, even if a linear stability result as in (R4) holds, we show that 
in a small-world setting we can use the graph parameters to move an eigenvalue from 
the linear analysis arbitrarily close to zero (see Section~\ref{Sec:instability}). 
\end{itemize}

Obviously, interpreting (R5) in practical terms means that external or internal 
small noise perturbations or going back to a finite-size case, may destabilize the grid, at
least on finite time scales. Hence, we provided a new tool to rigorously analyse the 
relation between small-world topology structures and the potential lack of practical 
stability of self-synchronized smart grids.

The paper is structured into two parts as our results (R1)--(R3) are actually very
general tools to deal with second-order ODEs on graphons, while the stability analysis
(R4)-(R5) focuses on our key application~\eqref{eq:Braess:0}. 

The first part of this article contains the proofs of the two main results, 
\cref{Thm:continuum:limit,Thm:continuum:limit:random}. The general strategy 
follows~\cite{Med14,KaM17} while several novel adaptations are required due to the 
fact that we are considering equations, which are of second order in time. Moreover, for 
the random case, it is a delicate step to consider a discrete averaged system as a technical 
tool to obtain asymptotically close solutions in probability to be able to eventually pass 
to the continuum limit as $N\to\infty$, which together is going to yield the main statements. 
In the second part of this article, the continuum limit is exploited to prove (R4)--(R5) 
using linearisation and Fourier techniques as well as graphons representing small-world 
random graphs. In this context, two new conserved quantities are identified and also the 
relation between topology and stability for self-synchronized power grids on small-world
topologies is studied.     

\section{Model Derivation}
\label{Sec:derivemodel}

For completeness, we include here the derivation of our main model~\eqref{eq:Braess:0},
which is very frequently used in recent power grid applications~\cite{WiT12}. Consider
$N$ electro-mechanical oscillators/rotators, which can be interpreted as producers
(e.g.~turbines) and consumers (e.g.~motors) depending upon the electric power $P_k$
generated ($P_k>0$) or consumed ($P_k<0$). The state variables of each oscillator 
are given by its mechanical phase angle $\theta_k=\theta_k(t)\in \mathbb{S}^1=
\mathbb{R}/\mathbb{Z}$ and its velocity $\dd{\theta}_k/\dd{t}=\dot{\theta}_k\in\R$. 
To have a synchronized
grid, one wants to have a common frequency $\mathfrak{F}=2 \pi\times 50 \textnormal{Hz}$ 
or $\mathfrak{F}=2 \pi\times 60 \textnormal{Hz}$. So it is natural to introduce the 
phase differences $\phi_k$ to the reference frequency by
\begin{equation}
\phi_k(t):=\theta_k(t)-\mathfrak{F} t, 
\end{equation}  
where we always have $k\in\{1,2,\ldots,N\}=:[N]$. Deriving equations for the phase 
differences $\phi_k=\phi_k(t)$, we invoke energy conservation
\begin{equation}
\label{eq:powerbalance}
P_{\textnormal{source},k}=P_{\textnormal{diss},k}+P_{\textnormal{acc},k}
+\sum_{k=1}^N P_{\textnormal{trans},\ell k},
\end{equation}
where we have balanced generated/consumed (or source) power with dissipated, accumulated
and transmitted power. The easiest term is dissipation, which is just given by 
$P_{\textnormal{diss},k}=\kappa_k(\dot{\theta}_k)^2$, where $\kappa_k$
is the friction coefficient of the $k$-th oscillator. The kinetic energy of each oscillator is 
$E_{\textnormal{kin},k}=\mathcal{I}_k (\dot{\theta}_k)^2/2$, where 
$\mathcal{I}_k$ is the moment of inertia. Accumulated power is then the time
derivative $P_{\textnormal{acc},k}=\dd{E_{\textnormal{kin},k}}/\dd{t}$. Next, we
assume that power transmission only depends upon the relative phase difference
$\theta_\ell-\theta_k$ via the leading-order term of the Fourier expansion of the
transmission function, i.e., transmission is proportional to $\sin(\theta_\ell-\theta_k)$.
Incorporating the maximal capacity $P_{\max,\ell k}$ of a transmission line we get
\begin{equation}
P_{\textnormal{trans},\ell k} = P_{\max,\ell k}\sin(\theta_\ell-\theta_k).
\end{equation}
Now one can insert the individual power terms into~\eqref{eq:powerbalance}. Under
the assumption that phase changes are slow variables compared to the reference 
frequency, $|\dot{\theta}_k|\ll \mathfrak{F}$, we may view terms of the
form $(\dot{\phi}_k)^2$ and $\dot{\phi}_k\ddot{\phi}_k$ as higher-order
terms, which then leads to 
\begin{equation}
\label{eq:physicalmodel}
\mathcal{I}_k\mathfrak{F}\ddot{\phi}_k = P_{\textnormal{source},k}-
\kappa_k \mathfrak{F}^2-2\kappa_k \mathfrak{F}\dot{\phi}_kl+
\sum_{\ell=1}^N P_{\max,\ell k}\sin(\theta_\ell-\theta_k).
\end{equation}
The model in the current form contains many free parameters and the basic case
studied in most applications first is to consider equal moment of inertia $\mathcal{I}$ 
and equal friction coefficient $\kappa$ for each machine. Using a non-dimensionalisation 
via
\begin{equation}
P_k:=\frac{P_{\textnormal{source},k}-\kappa \mathfrak{F}^2}{\mathcal{I}\mathfrak{F}^2},\qquad
\alpha:=2\kappa \mathcal{I},\qquad K_{k,\ell}:=\frac{P_{\max,\ell k}}{\mathcal{I}\mathfrak{F}}
\end{equation}
we indeed obtain~\eqref{eq:Braess:0} from \eqref{eq:physicalmodel}. Although we made an assumption on
the parameters, we are going to later on allow for randomness in $K_{k,\ell}$ 
so that parametric uncertainty is incorporated into the model making it even more realistic
than many standard formulations used frequently in applications. We are particularly interested 
in situations, where not all oscillators are coupled to each other but instead the machines 
are connected via a graph of $N$ nodes and the kernel $K_{k,\ell}$ characterizes if a connection 
exists between two nodes or not. Precisely, if the nodes $k$ and $\ell$ are connected, we 
have $K_{k,\ell}>0$ while $K_{k,\ell}=0$ corresponds to the case where $k$ and $\ell$ are 
uncoupled. 

We remark that one further common simplification arises if we have $P_{k}\equiv P$ for 
all $k\in\{1,\ldots N\}$. In this situation, by means of the rescaling $\phi_{k}\mapsto 
\phi_{k}+(Pt)/\alpha$ we can rewrite~\eqref{eq:Braess:0} as
\begin{equation}
\label{eq:Braess:1}
 \frac{\dd^2}{\dt^2}\phi_{k}=-\alpha \frac{\dd}{\dt}\phi_{k}+
\sum_{\ell=1}^{N}K_{k,\ell}\sin(\phi_{\ell}-\phi_{k}).
\end{equation}
One immediately checks that a stationary steady state for~\eqref{eq:Braess:1} is given by 
each constant $\Phi_{k}\equiv \Phi$ with $\Phi\in\R$.

\section{The Large Graph Limiting Model}
\label{Sec:model}

In this section, we present the background on graph convergence which is necessary for our studies and moreover, we precisely state the conditions under which we are able to pass to the continuum limit in~\eqref{eq:ODEintro}.

\subsection{Deterministic graphs and graph limits}
\label{Sec:deterministic:graphs}

To emphasise the connection to the underlying graph structure as announced in \cref{Sec:Introduction,Sec:derivemodel} let us view~\eqref{eq:Braess:1} from a slightly different perspective. More precisely, following the notation in~\cite{Med14, Med14a, Med14b} we denote by $G_{N}$ the graph of $N$ nodes which describes the coupling of the oscillator system. In particular, $G_{N}=\langle V(G_{N}), E(G_{N}) \rangle$, where $V(G_{N})$ is the set of nodes/vertices and $E(G_{N})$ the set of edges of the graph. Without loss of generality, we can identify $V(G_{N})=[N]$ where we adopt the notation $[N]\vcc= \{1,\ldots N\}$. Then $(k,\ell)\in E(G_{N})$ if and only if there exists a direct link between the nodes $k$ and $\ell$, where we identify the set of edges with the corresponding subset of $[N]\times [N]$. With this terminology, \eqref{eq:Braess:1} can be written as
\begin{equation}\label{eq:Braess:2}
 \frac{\dd^2}{\dt^2}\phi_{k}=-\alpha \frac{\dd}{\dt}\phi_{k}+\sum_{\ell\colon (k,\ell)\in E(G_{N})}K_{k,\ell}\sin(\phi_{\ell}-\phi_{k}).
\end{equation}
As already mentioned, we want to consider large networks, i.e.\@ the case where $1\ll N$ and to this end, we want to pass to the limit $N\to\infty$ in~\eqref{eq:Braess:2}. To be able to do this, we need the notion of \emph{graph limits} (see e.g.\@~\cite{Lovasz,LovaszSzegedy,Med14}). For convenience, we recall the basic ideas here. As noted above, the nodes of the graph are parametrised through $V(G_{N})=[N]$ while then each vertex $k\in [N]$ can equivalently be identified with the subset $[(k-1)/N,k/N]$ of the fixed reference interval $I=[0,1]$. 

Moreover, the information on the edges $E(G_{N})$ of the graph is encoded in a function $\K^{(N)}\colon I\times I\to \R$ which is defined as
\begin{equation}\label{eq:graph:sequence}
 \K^{(N)}(x,y)=K_{k,\ell} \quad \text{if} \quad (k,\ell)\in E(G_{N}) \quad \text{and}\quad (x,y)\in \Bigl[\frac{k-1}{N}, \frac{k}{N}\Bigr]\times \Bigl[\frac{\ell-1}{N}, \frac{\ell}{N}\Bigr]
\end{equation}
and $\K^{(N)}(x,y)=0$ otherwise. For our purpose it is then convenient to say that the graph $G_{N}$ converges to some limit object as $N\to \infty$ if $\K^{(N)}$ converges to a limit function $\K^{(\infty)}$ in a suitable topology.

\begin{remark}
 Classically \cite{LoS06}, the convergence of graphs is defined by means of the convergence of the homomorphism density
 \begin{equation*}
  t(F,G)=\frac{\mathrm{hom}(F,G_{N})}{\abs{V(G_{N})}^{\abs{V(F)}}}
 \end{equation*}
where $\mathrm{hom}(F,G_{N})$ denotes the number of homomorphisms from a fixed finite graph $F$ to $G_{N}$. More precisely, $G_{N}$ is said to converge, if $t(F,G_{N})$ converges as $N\to\infty$ for each finite simple graph $F$. However, since we will only consider situations where we have the stronger condition that $\K^{(N)}\to \K^{(\infty)}$ in $L^{2}(I\times I)$ we limit ourself to this notion of graph convergence.
\end{remark}

Conversely, it is also possible to start with the function $\K^{(\infty)}$ and construct a corresponding sequence of graphs $G_{N}=\langle V(G_{N}), E(G_{N})\rangle$ which converges to $\K^{(\infty)}$. This approach has been taken for example in~\cite{Med14, Med14a, KaM17} and we recall the construction here for completeness.

Precisely, for a given measurable, bounded and almost everywhere continuous function $\K^{(\infty)}\colon I\times I\to \R_{\geq 0}$ we construct a sequence of approximating \emph{weighted} graphs $G_{N}=\langle V(G_{N}), E(G_{N})\rangle$ by setting
\begin{equation}\label{eq:graph:inverse:1}
  V(G_{N})=[N]\quad \text{and}\quad E(G_{N})\vcc = \biggl\{(k,\ell)\in[N]\times [N]\;\bigg|\; \K^{(\infty)}\Bigl(\frac{k}{N},\frac{\ell}{N}\Bigr)>0\biggr\}
\end{equation}
and the corresponding weight for an edge $(k,\ell)\in E(G_{N})$ is given by $\K^{(\infty)}(k/N,\ell/N)$.
We then have the following result which states convergence of this graph sequence to the corresponding limit graphon $\K^{(\infty)}$ (see \cite{Med14}).

\begin{lemma}\label{Lem:graph:inverse:convergence}
 Let $\K^{(\infty)}\colon I\times I\to \R_{\geq 0}$ be a bounded, symmetric and almost every continuous function. Then the corresponding graph sequence $G_{N}=\langle V(G_{N}), E(G_{N})\rangle$ as constructed in~\eqref{eq:graph:inverse:1} converges to $\K^{(\infty)}$ in $L^{2}(I\times I)$, i.e.\@
 \begin{equation*}
  \lim_{N\to\infty}\norm{\K^{(N)}-\K^{(\infty)}}_{L^{2}(I\times I)}=0
 \end{equation*}
 where similarly to~\eqref{eq:graph:sequence}, $\K^{(N)}(x,y)=\K^{(\infty)}(k/N,\ell/N)$ if $(x,y)\in [(k-1)/N,k/N]\times [(\ell-1)/N,\ell/N]$ .
\end{lemma}

For completeness we include the short argument given by~\cite{Med14}

\begin{proof}
The construction of $\K^{(N)}$ yields $\K^{(N)}\to \K^{(\infty)}$ almost everywhere. Thus, due to the boundedness of $\K^{(\infty)}$ the claim follows from Lebesgue's dominated convergence theorem.
\end{proof}

\subsection{$\K$-random graphs}\label{Sec:random:graphs}

In this subsection, let us briefly recall the concept of $\K$-random graphs following mainly~\cite{KaM17} (see also~\cite{Med14a, Med14b}). The motivation for this is that we are also interested in oscillator networks whose underlying structure results from a stochastic process. More precisely, we are still interested in the model~\eqref{eq:Braess:1}, i.e.\@
\begin{equation}\label{eq:Braess:random:1}
 \frac{\dd^2}{\dt^2}\phi_{k}=-\alpha \frac{\dd}{\dt}\phi_{k}+\sum_{\ell=1}^{N}K_{k,\ell}(\omega)\sin(\phi_{\ell}-\phi_{k})
\end{equation}
but in contrast to the former, the coefficient matrix $K_{k,\ell}$ is now a random variable (as indicated by $\omega$). Moreover, in analogy to~\eqref{eq:graph:inverse:1} starting from a limiting graphon $\K^{(\infty)}\colon I\times I\to[0,1]$, we can construct sequences of so-called $\K^{(\infty)}$-random graphs $G_{N}(\omega)=\langle V(G_{N}), E(G_{N})\rangle$ satisfying
\begin{equation}\label{eq:graph:inverse:random:1}
  V(G_{N})=[N]\quad \text{and}\quad \P\bigl\{(k,\ell)\in E(G_{N})\bigr\}=\K^{(\infty)}\Bigl(\frac{k}{N},\frac{\ell}{N}\Bigr)\quad \text{for }k,\ell\in [N].
\end{equation} 
Here, the decision for each potential edge to be included in $E(G_{N})$ or not is made independently of the others. To make this construction more precise, we take for $N\in\N$ fixed the probability space
\begin{equation}\label{eq:random:coefficient:1}
 \bigl(\Omega_{N}=\{0,1\}^{N(N+1)/2}, 2^{\Omega_{N}},\P)
\end{equation}
from which we draw a random graph $G_{N}(\omega)=\langle V(G_{N}), E(G_{N})\rangle$. As explained in Section~\ref{Sec:deterministic:graphs} above, the latter is equivalent to determining the random coefficients $K_{k,\ell}^{(N)}(\omega)$ in~\eqref{eq:Braess:random:1} which are chosen to be independent Bernoulli random variables according to the condition
\begin{equation}\label{eq:random:coefficient:2}
 \E K_{k,\ell}^{(N)}(\omega)=\P\bigl(\{k,\ell\}\in E(G_{N})\bigr)=\K^{(\infty)}\Bigl(\frac{k}{N},\frac{\ell}{N}\Bigr)=\vcc \bar{K}_{k,\ell}^{(N)}.
\end{equation}

\subsection{Generalisations and the formal continuum limit}

Since it will not really cause any additional effort in deriving the continuum limit, instead of~\eqref{eq:Braess:1} or~\eqref{eq:graph:inverse:random:1}, we will consider the more general model
\begin{equation}\label{eq:ODE}
 \frac{\dd^{2}}{\dd{t}^2}\phi_{k}^{(N)}=-\alpha\frac{\dd}{\dd{t}}\phi_{k}^{(N)}+ \frac{1}{N}\sum_{\ell=1}^{N}K_{k,\ell}^{(N)}D(\phi_{\ell}^{(N)}-\phi_{k}^{(N)})+f(\phi_{k}^{(N)},t)\qquad \text{with }k\in[N].
\end{equation}
Here the coefficients $K_{k,\ell}^{(N)}$ may either be deterministic or stochastic. Moreover, our analysis allows for any $\alpha\in\R$. Note however, that $\alpha>0$ corresponds to the more physical situation of (positive) damping term whereas $\alpha>0$ leads to \emph{negative damping}.

\begin{remark}
 Note that we assume here that the coupling kernel $K_{k,\ell}$ is given by $\widetilde{K}_{k,\ell}/N$ which guarantees that the net field on each oscillator does not depend on $N$ (see for example~\cite[Section 5.4]{Kur84}). For simplicity, in the following we still write $K_{k,\ell}$ instead of $\widetilde{K}_{k,\ell}$.
 \end{remark}
 
Next, we define a function $\phi^{(N)}(x,t)$ with $x\in I$ through
\begin{equation}\label{eq:step:function}
 \phi^{(N)}(x,t)=\phi_{k}(t)\quad \text{if } \quad x\in\Bigl[\frac{k-1}{N}, \frac{k}{N}\Bigr)\quad \text{and}\quad k\in[N] 
\end{equation}
or equivalently $\phi^{(N)}(x,t)=\sum_{k=1}^{N}\phi_{k}(t)\chi_{[(k-1)/N,k/N]}$ where $\chi_{S}$ denotes the characteristic function of the set $S$.
With this function, \eqref{eq:ODE} can be rewritten as an integral equation, i.e.\@
\begin{multline}\label{eq:ODE:as:inteq}
  \frac{\dd^{2}}{\dd{t}^2}\phi^{(N)}(x,t)+\alpha\frac{\dd}{\dd{t}}\phi^{(N)}(x,t)\\*
  =\int_{I}\K^{(N)}(x,y)D\bigl(\phi^{(N)}(y,t)-\phi^{(N)}(x,t)\bigr)\dy+f\bigl(\phi^{(N)}(x,t),t\bigr).
\end{multline}
Assuming that $\phi^{(N)}$ converges to some $\phi(x,t)$ in a suitable sense, formally, we take the limit $N\to \infty$ in this equation. Thus, we expect that $\phi$ solves the integral equation
\begin{equation}\label{eq:cont:limit}
 \frac{\dd^{2}}{\dd{t}^2}\phi(x,t)+\alpha\frac{\dd}{\dd{t}}\phi(x,t)=\int_{I}\K^{(\infty)}(x,y)D\bigl(\phi(y,t)-\phi(x,t)\bigr)\dy+f\bigl(\phi(x,t),t\bigr)
\end{equation}
and $\phi$ denotes the \emph{continuum limit}.

\subsection{Assumptions and convergence to the continuum limit}

We collect in this section the basic assumptions on the non-linearities and we present the main results which we will show.

First of all, the non-linear functions $f$ and $D$ are assumed to be Lipschitz continuous. Precisely, this means that $D\colon \R \to \R$ and $f\colon \R\times[0,\infty)\to \R$ and there exist constants $L_{D}>0$ and $L_{f}>0$ such that
\begin{align}\label{eq:Lip:assumptions}
 \abs{D(u)-D(v)}\leq L_{D}\abs{u-v} \quad \text{and}\quad \sup_{t\in[0,T]}\abs{f(u,t)-f(v,t)}\leq L_{f}\abs{u-v}\quad \forall u,v\in\R.
\end{align}

The following proposition guarantees existence of solutions to the both equations~\eqref{eq:ODE} and~\eqref{eq:cont:limit}.

\begin{proposition}
\label{Prop:well:posed:cont:limit}
 Let $f$ and $D$ satisfy~\eqref{eq:Lip:assumptions} and let $g,h\in L^{\infty}(I)$ be given. Let furthermore $\K\in L^{\infty}(I\times I,\R_{\geq 0})$ and $\alpha\in\R$. Then the equation
 \begin{equation}\label{eq:cont:limit:general}
   \frac{\dd^{2}}{\dd{t}^2}\phi(x,t)+\alpha\frac{\dd}{\dd{t}}\phi(x,t)=\int_{I}\K(x,y)D\bigl(\phi(y,t)-\phi(x,t)\bigr)\dy+f\bigl(\phi(x,t),t\bigr)
 \end{equation}
 together with the initial conditions $\phi(\cdot, 0)=g$ and $\del_{t}\phi(\cdot,0)=h$ has a unique solution $\phi\in C^{2}(\R, L^{\infty}(I))$.
\end{proposition}

This result can be easily proven by a standard fixed-point argument. However, for completeness we give the main steps in Section~\ref{Sec:proof:existence}.

Our next main result states that the solution of~\eqref{eq:cont:limit}, i.e.\@ the continuum limit, in fact approximates the discrete system for large $N$. Yet, in order to be able to compare the solutions to the continuous and the discrete problem we also have to discretise the initial condition which is done by averaging over the different subintervals of $I$. Precisely, we define
 \begin{equation}\label{eq:averaged:initial:data}
  g_{k}^{(N)}\vcc = N\int_{(k-1)/N}^{k/N}g(x)\dx\qquad \text{and}\qquad h_{k}^{(N)}\vcc =N\int_{(k-1)/N}^{k/N}h(x)\dx.
 \end{equation}
 
With this, we can now precisely state the convergence of the solution to~\eqref{eq:ODE} to the continuum limit in the case of deterministic graphs.

\begin{theorem}\label{Thm:continuum:limit}
 Let $G_{N}=\langle V(G_{N}), E(G_{N})\rangle$ be a sequence of graphs with corresponding integral kernel $\K^{(N)}\colon I\times I\to \R_{\geq 0}$ which is uniformly bounded, i.e.\@ $\sup_{N\in\N}\norm{\K^{(N)}}_{L^{\infty}(I\times I)}\leq C$, and let $\K^{(\infty)}\in L^{\infty}(I\times I)$ such that we have $\norm{\K^{(N)}-\K^{(\infty)}}_{L^{2}(I\times I)}\to 0$ as $N\to\infty$. Moreover, for $g,h\in L^{\infty}(I)$ let $g_{k}^{(N)}$ and $h_{k}^{(N)}$ with $k\in[N]$ be given by~\eqref{eq:averaged:initial:data}. If $\phi_{k}^{(N)}\in C^{2}([0,\infty],\R)$ is the solution to~\eqref{eq:ODE} with initial condition $\phi_{k}^{(N)}(0)=g_{k}^{(N)}$ and $\frac{\dd}{\dt}\phi_{k}^{(N)}(0)=h_{k}^{(N)}$ and $\phi\in C^{2}([0,\infty),L^{\infty}(I))$ the solution to~\eqref{eq:cont:limit} with initial condition $\phi(\cdot,0)=g$ and $\frac{\dd}{\dt}\phi(\cdot,0)=h$ then we have
 \begin{equation*}
  \lim_{N\to\infty}\sup_{t\in[0,T]}\norm*{\phi(\cdot,t)-\sum_{k=1}^{N}\phi_{k}^{(N)}(t)\chi_{[(k-1)/N, k/N)}(\cdot)}_{L^{2}(I)}=0
 \end{equation*}
for each fixed $T>0$.
\end{theorem}

Additionally, we also have convergence in probability of solutions to~\eqref{eq:ODE} to the continuum limit in the case of random graphs.

\begin{theorem}\label{Thm:continuum:limit:random}
 Let $\K^{(\infty)}\in L^{\infty}(I\times I)$ be symmetric and almost everywhere continuous and let $G_{N}(\omega)=\langle V(G_{N}), E(G_{N})\rangle$ be a corresponding sequence of random graphs as constructed in Section~\ref{Sec:random:graphs} with independent coefficients $K_{k,\ell}^{(N)}(\omega)$. For $g,h\in L^{\infty}(I)$ let $g_{k}^{(N)}$ and $h_{k}^{(N)}$ with $k\in[N]$ be given by~\eqref{eq:averaged:initial:data}. If $u_{k}^{(N)}\in C^{2}([0,\infty],\R)$ is the solution to~\eqref{eq:ODE} with initial condition $u_{k}^{(N)}(0)=g_{k}^{(N)}$ and $\frac{\dd}{\dt}u_{k}^{(N)}(0)=h_{k}^{(N)}$ and $\phi\in C^{2}([0,\infty),L^{\infty}(I))$ the solution to~\eqref{eq:cont:limit} with initial condition $\phi(\cdot,0)=g$ and $\frac{\dd}{\dt}\phi(\cdot,0)=h$, then we have 
  \begin{equation*}
  \sup_{t\in[0,T]}\norm*{\phi(\cdot,t)-\sum_{k=1}^{N}u_{k}^{(N)}(t)\chi_{[(k-1)/N, k/N)}(\cdot)}_{L^{2}(I)}\longrightarrow 0\qquad \text{in probability as }N\to\infty
 \end{equation*}
 for each fixed $T>0$.
\end{theorem}

Let us finally note that in order to simplify the notation in certain places we use the common abbreviation
\begin{equation*}
 \dot{w}=\frac{\dd}{\dt}w \qquad \text{as well as}\qquad \ddot{w}=\frac{\dd^2}{\dt^2}w.
\end{equation*}

\section{The continuum limit for deterministic graphs}

In this section we will give the proof of Theorem~\ref{Thm:continuum:limit}. Before we come to this, we note the following lemma which states that the discretised initial data as given in~\eqref{eq:averaged:initial:data} converge to the original functions as $N\to\infty$.

 \begin{lemma}\label{Lem:convergence:initial:data}
  For $g,h\in L^{\infty}(I)$ let $g^{(N)}=(g_{k}^{(N)})_{k\in [N]}$ and  $h^{(N)}=(h_{k}^{(N)})_{k\in [N]}$ be given by~\eqref{eq:averaged:initial:data}. Then 
  \begin{equation*}
   g^{(N)}\to g \qquad \text{and}\qquad h^{(N)}\to h \qquad \text{in } L^{2}(I) \quad \text{as }N\to \infty.
  \end{equation*}
 \end{lemma}

 \begin{proof}
  The statement follows immediately from dominated convergence together with the Lebesgue differentiation theorem.
 \end{proof}
 
 We are now prepared to prove Theorem~\ref{Thm:continuum:limit} which follows ideas from~\cite{Med14}.

\begin{proof}[Proof of Theorem~\ref{Thm:continuum:limit}]
 We rewrite~\eqref{eq:ODE} as an integral equation, i.e.\@ recalling the definition of $\phi^{(N)}$ in~\eqref{eq:step:function} we have $\phi^{(N)}(x,t)=\sum_{k=1}^{N}\phi_{k}^{(N)}(t)\chi_{[(k-1)/N, k/N)}(x)$. Then, as seen in~\eqref{eq:ODE:as:inteq} we obtain
 \begin{multline}\label{eq:Galerkin:ODE}
  \frac{\dd^{2}}{\dd{t}^2}\phi^{(N)}(x,t)+\alpha\frac{\dd}{\dd{t}}\phi^{(N)}(x,t)\\*
  = \int_{I}\K^{(N)}(x,y)D\bigl(\phi^{(N)}(y,t)-\phi^{(N)}(x,t)\bigr)\dy+f\bigl(\phi^{(N)}(x,t),t\bigr)
 \end{multline}
 equipped with the initial conditions $\phi^{(N)}(\cdot,0)=g^{(N)}\vcc= \sum_{k=1}^{N}g_{k}^{(N)}(t)\chi_{[(k-1)/N, k/N)}(x)$ and $\del_{t}\phi^{(N)}(\cdot,0)=h^{(N)}\vcc= \sum_{k=1}^{N}h_{k}^{(N)}(t)\chi_{[(k-1)/N, k/N)}(x)$.
 
 To prove that the sequence of solutions $\phi^{(N)}$ to~\eqref{eq:Galerkin:ODE} converges to the solution $\phi$ for~\eqref{eq:cont:limit} we define their difference to be $\rho^{(N)}(x,t)\vcc= \phi^{(N)}(x,t)-\phi(x,t)$. Thus, taking also the difference of the corresponding equations~\eqref{eq:Galerkin:ODE} and~\eqref{eq:cont:limit}, we find that $\rho^{(N)}$ solves
 \begin{multline*}
  \frac{\dd^{2}}{\dd{t}^2}\rho^{(N)}+\alpha\frac{\dd}{\dd{t}}\rho^{(N)}= \int_{I}\K^{(N)}(x,y)\Bigl(D\bigl(\phi^{(N)}(y,t)-\phi^{(N)}(x,t)\bigr)-D\bigl(\phi(y,t)-\phi(x,t)\bigr)\Bigr)\dy\\*
  +\int_{I}\bigl(\K^{(N)}(x,y)-\K^{(\infty)}(x,y)\bigr)D\bigl(\phi(y,t)-\phi(x,t)\bigr)\dy+f\bigl(\rho^{(N)}(x,t),t\bigr)-f\bigl(\phi(x,t),t\bigr).
 \end{multline*}
 The last equation can be written equivalently as system of two first order equations as
 \begin{multline*}
  \dot{\rho}^{(N)}=\sigma^{(N)}\\*
  \shoveleft{\dot{\sigma}^{(N)}=-\alpha\sigma^{(N)}+\int_{I}\K^{(N)}(x,y)\Bigl(D\bigl(\phi^{(N)}(y,t)-\phi^{(N)}(x,t)\bigr)-D\bigl(\phi(y,t)-\phi(x,t)\bigr)\Bigr)\dy}\\*
  +\int_{I}\bigl(\K^{(N)}(x,y)-\K^{(\infty)}(x,y)\bigr)D\bigl(\phi(y,t)-\phi(x,t)\bigr)\dy+f\bigl(\phi^{(N)}(x,t),t\bigr)-f\bigl(\phi(x,t),t\bigr).
 \end{multline*}
 We now multiply the first equation by $\rho^{(N)}$ and the second one by $\sigma^{(N)}$ and we integrate over $I$ which yields
 \begin{multline*}
  \frac{1}{2}\frac{\dd}{\dt}\int_{I}\bigl(\rho^{(N)}(x,t)\bigr)^2\dx=\int_{I}\sigma^{(N)}(x,t)\rho^{(N)}(x,t)\dx\\*
  \shoveleft{\frac{1}{2}\frac{\dd}{\dt}\int_{I}\bigl(\sigma^{(N)}(x,t)\bigr)^2\dx=-\alpha\int_{I}\bigl(\sigma^{(N)}(x,t)\bigr)^2\dx}\\*
  +\int_{I^2}\K^{(N)}(x,y)\Bigl(D\bigl(\phi^{(N)}(y,t)-\phi^{(N)}(x,t)\bigr)-D\bigl(\phi(y,t)-\phi(x,t)\bigr)\Bigr)\sigma^{(N)}(x,t)\dy\dx\\*
  +\int_{I^2}\bigl(\K^{(N)}(x,y)-\K^{(\infty)}(x,y)\bigr)D\bigl(\phi(y,t)-\phi(x,t)\bigr)\sigma^{(N)}(x,t)\dy\dx\\*
  +\int_{I}\Bigl(f\bigl(\phi^{(N)}(x,t),t\bigr)-f\bigl(\phi(x,t),t\bigr)\Bigr)\sigma^{(N)}(x,t)\dx.
 \end{multline*}
 Adding the two equations, exploiting the Lipschitz continuity of $D$ and $f$ and applying Cauchy's inequality, we obtain
 \begin{multline*}
  \frac{1}{2}\frac{\dd}{\dt}\bigl(\norm{\rho^{(N)}}_{L^{2}(I)}^2+\norm{\sigma^{(N)}}_{L^{2}(I)}^2\bigr)\leq \norm{\rho^{(N)}}_{L^{2}(I)}\norm{\sigma^{(N)}}_{L^{2}(I)}-\alpha \norm{\rho^{(N)}}_{L^{2}(I)}^2\\*
  +L_{D}\norm{\K^{(N)}}_{L^{\infty}}\int_{I^2}\abs{\rho^{(N)}(y,t)-\rho^{(N)}(x,t)}\abs{\sigma^{(N)}(x,t)}\dx\dy\\*
  +C\norm{\K^{(\infty)}-\K^{(N)}}_{L^{2}(I^2)}\norm{\sigma^{(N)}}_{L^{2}(I)}+L_{f}\int_{I}\abs{\rho^{(N)}(x,t)}\abs{\sigma^{(N)}(x,t)}\dx.
 \end{multline*}
Note that we also used that $(x,y)\mapsto D(\phi(y,t)-\phi(x,t))$ is uniformly bounded on $I\times I$ due to Proposition~\ref{Prop:well:posed:cont:limit}. Another application of Cauchy's inequality yields
 \begin{multline*}
  \frac{1}{2}\frac{\dd}{\dt}\bigl(\norm{\rho^{(N)}}_{L^{2}(I)}^2+\norm{\sigma^{(N)}}_{L^{2}(I)}^2\bigr)\leq \bigl(2L_{D}\norm{\K^{(N)}}_{L^{\infty}}+L_{f}+1\bigr)\norm{\rho^{(N)}}_{L^{2}(I)}\norm{\sigma^{(N)}}_{L^{2}(I)}\\*
  -\alpha \norm{\rho^{(N)}}_{L^{2}(I)}^2+C\norm{\K^{(\infty)}-\K^{(N)}}_{L^{2}(I^2)}\norm{\sigma^{(N)}}_{L^{2}(I)}.
 \end{multline*}
Together with Young's inequality we thus obtain 
 \begin{multline*}
  \frac{1}{2}\frac{\dd}{\dt}\bigl(\norm{\rho^{(N)}}_{L^{2}(I)}^2+\norm{\sigma^{(N)}}_{L^{2}(I)}^2\bigr)\\*
  \leq \frac{1}{2}\bigl(2L_{D}\norm{\K^{(N)}}_{L^{\infty}}+L_{f}+1+2\abs{\alpha}\bigr)\bigl(\norm{\rho^{(N)}}_{L^{2}(I)}^2+\norm{\sigma^{(N)}}_{L^{2}(I)}^2\bigr)\\*
  +\frac{C}{2}\bigl(\norm{\K^{(\infty)}-\K^{(N)}}_{L^{2}(I^2)}^2+ \norm{\sigma^{(N)}}_{L^{2}(I)}^2\bigr).
 \end{multline*}
 Estimating the right-hand side further, we end up with
  \begin{multline*}
  \frac{1}{2}\frac{\dd}{\dt}\bigl(\norm{\rho^{(N)}}_{L^{2}(I)}^2+\norm{\sigma^{(N)}}_{L^{2}(I)}^2\bigr)\\*
  \leq \frac{1}{2}\bigl(2L_{D}\norm{\K^{(N)}}_{L^{\infty}}+L_{f}+1+2\abs{\alpha}+C\bigr)\bigl(\norm{\rho^{(N)}}_{L^{2}(I)}^2+\norm{\sigma^{(N)}}_{L^{2}(I)}^2\bigr)\\*
  +\frac{C}{2}\norm{\K^{(\infty)}-\K^{(N)}}_{L^{2}(I^2)}^2.
 \end{multline*}
 Grönwall's inequality finally yields
 \begin{multline*}
  \sup_{t\in[0,T]}\bigl(\norm{\rho^{(N)}}_{L^{2}(I)}^2+\norm{\sigma^{(N)}}_{L^{2}(I)}^2\bigr)\\*
  \leq \bigl(\norm{g^{(N)}-g}_{L^{2}(I)}^2+\norm{h^{(N)}-h}_{L^{2}(I)}^2+C_{1}\norm{\K^{(\infty)}-\K^{(N)}}_{L^{2}(I^2)}^2\bigr)\exp(C_{2}T)
 \end{multline*}
 with constants $C_{1},C_{2}>0$. Due to Lemma~\ref{Lem:convergence:initial:data} and the assumptions $\lim_{N\to\infty}\norm{\K^{(N)}-\K^{(\infty)}}_{L^{2}(I\times I)}=0$, for $T>0$ fixed, the right-hand side tends to zero as $N\to \infty$. This then finishes the proof since $\rho^{(N)}=\phi^{N}-\phi$.
\end{proof}

\section{Approximation for $\K$-random graphs}

In this section we consider~\eqref{eq:ODE} on a $\K$-random graph and we show that the corresponding solutions, for large $N$, can be approximated by the corresponding averaged system. The proofs in this section are mainly motivated by ideas from~\cite{KaM17}.

\subsection{The random graph model}

For convenience, we recall the equations we are considering in this section, namely~\eqref{eq:ODE} with random coefficients $K_{k,\ell}^{(N)}(\omega)$, i.e.\@
\begin{equation}\label{eq:ODE:random}
 \ddot{u}_{k}^{(N)}(t,\omega)=-\alpha \dot{u}_{k}^{(N)}(t,\omega)+\frac{1}{N}\sum_{\ell=1}^{N}K_{k,\ell}^{(N)}(\omega)D\bigl(u_{\ell}^{(N)}-u_{k}^{(N)}\bigr)+f\bigl(u_{k}^{(N)},t\bigr)
\end{equation}
and the corresponding averaged equation which reads
\begin{equation}\label{eq:ODE:averaged}
 \ddot{\phi}_{k}^{(N)}(t)=-\alpha \dot{\phi}_{k}^{(N)}(t)+\frac{1}{N}\sum_{\ell=1}^{N}\bar{K}_{k,\ell}^{(N)}D\bigl(\phi_{\ell}^{(N)}-\phi_{k}^{(N)}\bigr)+f\bigl(\phi_{k}^{(N)},t\bigr).
\end{equation}
Here $\bar{K}_{k,\ell}^{(N)}=\E K_{k,\ell}^{(N)}(\omega)$. Let us furthermore introduce the following notation. For $N\in\N$ fixed, we denote by $\phi^{(N)}$ and $u^{(N)}$ the vectors
\begin{equation*}
 \phi^{(N)}=\bigl(\phi_{1}^{(N)},\ldots, \phi_{N}^{(N)}\bigr)\quad \text{and}\quad u^{(N)}=\bigl(u_{1}^{(N)},\ldots, u_{N}^{(N)}\bigr) \quad \text{in }\R^{N}
\end{equation*}
and we denote by $\norm{\cdot}_{2}$ the norm on $\R^{N}$ given by
\begin{equation*}
 \norm{v}_{2}=\biggl(\frac{1}{N}\sum_{k=1}^{N}v_{k}^{2}\biggr)^{1/2}
\end{equation*}
where $v=(v_{1},\ldots, v_{N})$. Then, we have the following result which states that solutions to~\eqref{eq:ODE:random} and~\eqref{eq:ODE:averaged} are asymptotically stable for large $N$.

\begin{proposition}\label{Prop:approximation:averaged}
 Let the initial conditions $u_{k}^{(N)}(0)$ and $\phi_{k}^{(N)}(0)$ be uniformly bounded with respect to $k$ and $N$ and assume that
 \begin{equation*}
  \lim_{N\to \infty}\norm[\big]{u^{(N)}(0)-\phi^{(N)}(0)}_{2}=0.
 \end{equation*}
 Then, we have
 \begin{equation*}
  \sup_{t\in[0,T]}\norm[\big]{u^{(N)}(t)-\phi^{(N)}(t)}_{2}\longrightarrow 0\qquad \text{in probability as }N\to\infty
 \end{equation*}
 where $u^{(N)}$ and $\phi^{(N)}$ are the solutions to~\eqref{eq:ODE:random} and~\eqref{eq:ODE:averaged} respectively.
\end{proposition}

As part of the proof we will need the following lemma which is also contained in \cite[Lemma~4.2]{KaM17} and which we present here in slightly rephrased form. 

\begin{lemma}\label{Lem:convergence:probability}
 Let $T>0$ and $a^{(N)}=(a_{k,\ell}^{(N)})\colon [0,T]\to \R^{N\times N}$ deterministic such that
 \begin{equation*}
  \sup_{t\in[0,T]}\max_{(k,\ell)\in[N]\times [N]}\abs[\big]{a_{k,\ell}^{(N)}(t)}\leq C\quad \text{for all } N\in\N.
 \end{equation*}
 Let furthermore $\mu^{(N)}(t)=(\mu_{1}^{(N)}(t),\ldots, \mu_{N}^{(N)}(t))$ with
 \begin{equation*}
  \mu_{k}^{(N)}(t)=\frac{1}{N}\sum_{\ell=1}^{N}a_{k,\ell}^{(N)}(t)\bigl(\bar{K}_{k,\ell}^{(N)}-K_{k,\ell}^{(N)}(\omega)\bigr) \quad \text{for }k\in[N] \text{ and }t\in[0,T].
 \end{equation*}
 Then, we have
 \begin{equation*}
  \lim_{N\to\infty}\int_{0}^{T}\norm{\mu^{(N)}(t)}_{2}^{2}\dt \longrightarrow 0 \qquad \text{in probability as }N\to\infty.
 \end{equation*}
\end{lemma}

Since we assume slightly different conditions on the graphon $\K^{(\infty)}$ we recall the proof from~\cite{KaM17} here for completeness.

\begin{proof}[Proof of Lemma~\ref{Lem:convergence:probability}]
 Using the definition of $\mu_{k}^{(N)}$ we can rewrite
 \begin{multline}\label{eq:convergence:probability:1}
  \E \int_{0}^{T}\bigl(\mu_{k}^{(N)}(t)\bigr)^{2}\dt\\*
  =\frac{1}{N^2}\E\biggl(\sum_{\ell,m=1}^{N}\int_{0}^{T}a_{k,\ell}^{(N)}(t)a_{k,m}^{(N)}(t)\dt\bigl(\bar{K}_{k,\ell}^{(N)}-K_{k,\ell}^{(N)}(\omega)\bigr)\bigl(\bar{K}_{k,m}^{(N)}-K_{k,m}^{(N)}(\omega)\bigr)\biggr).
 \end{multline}
 Since the $a_{k,\ell}^{(N)}$ are deterministic, we can pass the expectation to the expression $\bigl(\bar{K}_{k,\ell}^{(N)}-K_{k,\ell}^{(N)}(\omega)\bigr)\bigl(\bar{K}_{k,m}^{(N)}-K_{k,m}^{(N)}(\omega)\bigr)$ and due to the independence of $K_{k,\ell}^{(N)}$ and $K_{k,m}^{(N)}$ for $\ell\neq m$ we have
 \begin{equation*}
  \E\Bigl(\bigl(\bar{K}_{k,\ell}^{(N)}-K_{k,\ell}^{(N)}(\omega)\bigr)\bigl(\bar{K}_{k,m}^{(N)}-K_{k,m}^{(N)}(\omega)\bigr)\Bigr)=\E\Bigl(\bigl(\bar{K}_{k,\ell}^{(N)}-K_{k,\ell}^{(N)}(\omega)\bigr)^2\Bigr)\delta_{\ell,m}.
 \end{equation*}
 Using this in~\eqref{eq:convergence:probability:1}, we obtain
  \begin{multline}\label{eq:convergence:probability:2}
  \E \int_{0}^{T}\bigl(\mu_{k}^{(N)}(t)\bigr)^{2}\dt=\frac{1}{N^2}\sum_{\ell=1}^{N}\int_{0}^{T}\bigl(a_{k,\ell}^{(N)}(t)\bigr)^2\dt\E\Bigl(\bigl(\bar{K}_{k,\ell}^{(N)}-K_{k,\ell}^{(N)}(\omega)\bigr)^2\Bigr)\\*
  =\frac{1}{N^2}\sum_{\ell=1}^{N}\int_{0}^{T}\bigl(a_{k,\ell}^{(N)}(t)\bigr)^2\dt\Bigl(\bigl(\bar{K}_{k,\ell}^{(N)}\bigr)^2-\E\bigl(K_{k,\ell}^{(N)}(\omega)\bigr)^2\Bigr).
 \end{multline}
 The last factor on the right-hand side is just the variance of the Bernoulli random variables $K_{k,\ell}^{(N)}$ which is given by $\bar{K}_{k,\ell}^{(N)}(1-\bar{K}_{k,\ell}^{(N)})\leq 1$. Consequently, together with the assumptions on $a_{k,\ell}^{(N)}$, we can estimate the right-hand side of~\eqref{eq:convergence:probability:2} as
   \begin{equation}\label{eq:convergence:probability:3}
  \E \int_{0}^{T}\bigl(\mu_{k}^{(N)}(t)\bigr)^{2}\dt\leq \frac{1}{N^2}C^2TN=\frac{C^2T}{N}.
 \end{equation}
 To conclude the proof, we note that the Markov inequality together with~\eqref{eq:convergence:probability:3} yields
 \begin{equation*}
  \P\biggl(\int_{0}^{T}\norm[\big]{\mu^{(N)}(t)}_{2}^{2}\geq \eps\biggr)=\P\biggl(\frac{1}{N}\sum_{k=1}^{N}\int_{0}^{T}\bigl(\mu_{k}^{(N)}(t)\bigr)^2\dt\geq \eps\biggr)\leq C^{2}T(\eps N)^{-1}\to 0
 \end{equation*}
as $N\to\infty$ for any $\eps>0$ fixed. This finishes the proof.
\end{proof}

We are now prepared to prove Proposition~\ref{Prop:approximation:averaged}.

\begin{proof}[Proposition~\ref{Prop:approximation:averaged}]
 As a first step, we rewrite~\cref{eq:ODE:random,eq:ODE:averaged} as systems of first order equations, i.e.\@
 
 \begin{equation*}
  \begin{aligned}
   \dot{u}_{k}^{(N)}&=v_{k}^{(N)}\\
   \dot{v}_{k}^{(N)}&=-\alpha v_{k}^{(N)}+\frac{1}{N}\sum_{\ell}^{N}K_{k,\ell}^{(N)}D\bigl(u_{\ell}^{(N)}-u_{k}^{(N)}\bigr)+f\bigl(u_{k}^{(N)},t\bigr)
  \end{aligned}
 \end{equation*}
and 
 \begin{equation*}
  \begin{aligned}
   \dot{\phi}_{k}^{(N)}&=\psi_{k}^{(N)}\\
   \dot{\psi}_{k}^{(N)}&=-\alpha \psi_{k}^{(N)}+\frac{1}{N}\sum_{\ell=1}^{N}\bar{K}_{k,\ell}^{(N)}D\bigl(\phi_{\ell}^{(N)}-\phi_{k}^{(N)}\bigr)+f\bigl(\phi_{k}^{(N)},t\bigr).
  \end{aligned}
 \end{equation*}
Next, taking the difference of these two systems and denoting $\rho_{k}^{(N)}\vcc=u_{k}^{(N)}-\phi_{k}^{(N)}$ and $\sigma_{k}^{(N)}\vcc=\dot{\rho}_{k}^{(N)}$, we obtain upon rewriting that
\begin{equation*}
 \begin{aligned}
  \dot{\rho}_{k}^{(N)}&=\sigma_{k}^{(N)}\\
  \dot{\sigma}_{k}^{(N)}&=-\alpha \sigma_{k}^{(N)}+\frac{1}{N}\sum_{\ell=1}^{N}K_{k,\ell}^{(N)}(\omega)\Bigl(D(u_{\ell}^{(N)}-u_{k}^{(N)})-D(\phi_{\ell}^{(N)}-\phi_{k}^{(N)})\Bigr)\\
  &\qquad+\frac{1}{N}\sum_{\ell=1}^{N}\Bigl(K_{k,\ell}^{(N)}-\bar{K}_{k,\ell}^{(N)}\Bigr)D\bigl(\phi_{\ell}^{(N)}-\phi_{k}^{(N)}\bigr)\\
  &\qquad\qquad\qquad+f\bigl(u_{k}D\bigl(\phi_{\ell}^{(N)}-\phi_{k}^{(N)}\bigr)+f\bigl(\phi_{k}^{(N)},t\bigr)-f\bigl(\phi_{k}^{(N)},t\bigr).
 \end{aligned}
\end{equation*}
Multiplying the first equation by $\rho_{k}^{(N)}/N$ and the second one by $\sigma_{k}^{(N)}/N$ and summing over $k\in[N]$ we find
\begin{equation*}
 \begin{aligned}
  \frac{1}{2}\frac{\dd}{\dt}\norm{\rho^{(N)}}_{2}^{2}&=\frac{1}{N}\sum_{k=1}^{N}\rho_{k}^{(N)}\sigma_{k}^{(N)}\\
  \frac{1}{2}\frac{\dd}{\dt}\norm{\sigma^{(N)}}_{2}^{2}&=-\alpha \norm{\sigma^{(N)}}_{2}^{2}\\
  &\qquad+\frac{1}{N^2}\sum_{k,\ell=1}^{N}K_{k,\ell}^{(N)}(\omega)\Bigl(D(u_{\ell}^{(N)}-u_{k}^{(N)})-D(\phi_{\ell}^{(N)}-\phi_{k}^{(N)})\Bigr)\sigma_{k}^{(N)}\\
  &\qquad\qquad +\frac{1}{N^2}\sum_{k,\ell=1}^{N}\Bigl(K_{k,\ell}^{(N)}-\bar{K}_{k,\ell}^{(N)}\Bigr)D\bigl(\phi_{\ell}^{(N)}-\phi_{k}^{(N)}\bigr)\sigma_{k}^{(N)}\\
  &\qquad\qquad\qquad\qquad+\frac{1}{N}\sum_{k=1}^{N}\Bigl(f\bigl(\phi_{k}^{(N)},t\bigr)-f\bigl(\phi_{k}^{(N)},t\bigr)\Bigr)\sigma_{k}^{(N)}.
 \end{aligned}
\end{equation*}
Using the elementary inequality $ab\leq a^2/2+b^2/2$ for the first equation and the Lipschitz continuity of $D$ and $f$ for the second one, we can estimate as
\begin{equation}\label{eq:approximation:averaged:1}
 \begin{aligned}
   \frac{1}{2}\frac{\dd}{\dt}\norm{\rho^{(N)}}_{2}^{2}&\leq \frac{1}{2}\norm{\rho^{(N)}}_{2}^{2}+ \frac{1}{2}\norm{\sigma^{(N)}}_{2}^{2}\\
   \frac{1}{2}\frac{\dd}{\dt}\norm{\sigma^{(N)}}_{2}^{2}&\leq -\alpha \norm{\sigma^{(N)}}_{2}^{2}+\frac{L_{D}}{N^2}\sum_{k,\ell=1}^{N}K_{k,\ell}^{(N)}(\omega)\abs{\rho_{\ell}^{(N)}-\rho_{k}^{(N)}}\abs{\sigma_{k}^{(N)}}\\
   &\qquad +\frac{1}{N^2}\sum_{k=1}^{N}\biggl(\sum_{\ell=1}^{N}\Bigl(K_{k,\ell}^{(N)}-\bar{K}_{k,\ell}^{(N)}\Bigr)D\bigl(\phi_{\ell}^{(N)}-\phi_{k}^{(N)}\bigr)\biggr)\sigma_{k}^{(N)}\\
   &\qquad\qquad+\frac{L_{f}}{N}\sum_{k=1}^{N}\abs{\rho_{k}^{(N)}}\abs{\sigma_{k}^{(N)}}.
 \end{aligned}
\end{equation}
Using again $ab\leq a^2/2+b^2/2$ as well as $K_{k,\ell}^{(N)}\leq 1$ (see \cref{eq:random:coefficient:1,eq:random:coefficient:2}) we can further estimate the right-hand side to find
\begin{multline}\label{eq:approximation:averaged:2}
 \frac{1}{2}\frac{\dd}{\dt}\norm{\sigma^{(N)}}_{2}^{2}\leq -\alpha \norm{\sigma^{(N)}}_{2}^{2}+\frac{L_{D}}{N^2}\sum_{k,\ell=1}^{N}\Bigl(\frac{\abs{\rho_{\ell}^{(N)}}^2}{2}+\frac{\abs{\rho_{k}^{(N)}}^2}{2}+\abs{\sigma_{k}^{(N)}}^2\Bigr)\\*
 \shoveright{+\frac{1}{2N}\sum_{k=1}^{N}\biggl(\frac{1}{N}\sum_{\ell=1}^{N}\Bigl(K_{k,\ell}^{(N)}-\bar{K}_{k,\ell}^{(N)}\Bigr)D\bigl(\phi_{\ell}^{(N)}-\phi_{k}^{(N)}\bigr)\biggr)^{2}}\\*
 +\frac{1}{2}\norm{\sigma^{(N)}}_{2}^{2}+\frac{L_{f}}{2}\Bigl(\norm{\rho^{(N)}}_{2}^{2}+\norm{\sigma^{(N)}}_{2}^{2}\Bigr).
\end{multline}
If we define
\begin{equation*}
 a_{k,\ell}^{(N)}\vcc=D\bigl(\phi_{\ell}^{(N)}-\phi_{k}^{(N)}\bigr)\quad \text{and}\quad \mu_{k}^{(N)}\vcc=\frac{1}{N}\sum_{\ell=1}^{N}\Bigl(K_{k,\ell}^{(N)}-\bar{K}_{k,\ell}^{(N)}\Bigr)D\bigl(\phi_{\ell}^{(N)}-\phi_{k}^{(N)}\bigr)
\end{equation*}
we can finally rewrite the right-hand side of~\eqref{eq:approximation:averaged:2} as
\begin{equation}
 \frac{1}{2}\frac{\dd}{\dt}\norm{\sigma^{(N)}}_{2}^{2}\leq \Bigl(\frac{L_{f}}{2}+L_{D}\Bigr)\norm{\rho^{(N)}}_{2}^{2}+\Bigl(\frac{L_{f}+1}{2}+L_{D}-\alpha\Bigr)\norm{\sigma^{(N)}}_{2}^{2}+\frac{1}{2}\norm{\mu^{(N)}}_{2}^{2}.
\end{equation}
Using this estimate together with the estimate on $\frac{\dd}{\dt}\norm{\rho^{(N)}}_{2}^{2}$ from~\eqref{eq:approximation:averaged:1} we get
\begin{equation*}
 \frac{\dd}{\dt}\Bigl(\norm{\rho^{(N)}}_{2}^{2}+\norm{\sigma^{(N)}}_{2}^{2}\Bigr)\leq \Bigl(\frac{L_{f}}{2}+L_{D}+1+\abs{\alpha}\Bigr)\Bigl(\norm{\rho^{(N)}}_{2}^{2}+\norm{\sigma^{(N)}}_{2}^{2}\Bigr)+\frac{1}{2}\norm{\mu^{(N)}}_{2}^{2}.
\end{equation*}
By Grönwall's inequality we thus deduce
\begin{multline*}
 \sup_{t\in[0,T]}\Bigl(\norm{\rho^{(N)}}_{2}^{2}+\norm{\sigma^{(N)}}_{2}^{2}\Bigr)\\*
 \leq \biggl(\norm{\rho^{(N)}(0)}_{2}^{2}+\norm{\sigma^{(N)}(0)}_{2}^{2}+\frac{1}{2}\int_{0}^{T}\norm{\mu^{(N)}(s)}_{2}^{2}\ds\biggr)\ee^{\left(\frac{L_{f}}{2}+L_{D}+1+\abs{\alpha}\right)T}
\end{multline*}
By assumption $\norm{\rho^{(N)}(0)}_{2}^{2}+\norm{\sigma^{(N)}(0)}_{2}^{2}\to 0$ and Lemma~\ref{Lem:convergence:probability} yields $\int_{0}^{T}\norm{\mu^{(N)}(s)}_{2}^{2}\ds\to 0$ in probability as $N\to\infty$, which ends the proof.
\end{proof}

We are now prepared to give the proof of Theorem~\ref{Thm:continuum:limit:random} which is actually an immediate consequence of the previous results.

\begin{proof}[Proof of Theorem~\ref{Thm:continuum:limit:random}]
 By assumption $u^{(N)}=(u_{1}^{(N)},\ldots, u_{N}^{(N)})$ is the solution to~\eqref{eq:ODE} with kernel $K_{k,\ell}^{(N)}(\omega)$ and initial condition $u_{k}^{(N)}(0)=g_{k}^{(N)}$ and $\frac{\dd}{\dt}u_{k}^{(N)}(0)=h_{k}^{(N)}$ whereas $\phi\in C^{2}([0,\infty),L^{\infty}(I))$ solves~\eqref{eq:cont:limit} with initial condition $\phi(\cdot,0)=g$ and $\frac{\dd}{\dt}\phi(\cdot,0)=h$. Let now $\phi^{(N)}=(\phi_{1}^{(N)},\ldots,\phi_{N}^{(N)})$ be the solution to~\eqref{eq:ODE:averaged} with the same initial data as $u$, i.e.\@ $\phi_{k}^{(N)}(0)=g_{k}^{(N)}$ and $\frac{\dd}{\dt}\phi_{k}^{(N)}(0)=h_{k}^{(N)}$. Then, we have 
 \begin{multline}\label{eq:proof:cont:limit:stoch:1}
  \sup_{t\in[0,T]}\norm*{\phi(\cdot,t)-\sum_{k=1}^{N}u_{k}^{(N)}\chi_{[(k-1)/N,k/N]}(\cdot)}_{L^{2}(I)}\\*
  \leq \sup_{t\in[0,T]}\norm*{\phi(\cdot,t)-\sum_{k=1}^{N}\phi_{k}^{(N)}\chi_{[(k-1)/N,k/N]}(\cdot)}_{L^{2}(I)}+\sup_{t\in[0,T]}\norm*{\phi^{(N)}(t)(\cdot,t)-u^{(N)}(t)}_{2}.
 \end{multline}
 By construction, $\bar{K}_{k,\ell}^{(N)}=\E K_{k,\ell}^{(N)}(\omega)=\K^{(\infty)}(k/N,\ell/N)$ (see~\eqref{eq:random:coefficient:2}) such that the graphon sequence $\bar{\K}^{(N)}$ corresponding to $\bar{K}_{k,\ell}$ (see~\eqref{eq:graph:sequence}) converges to $\K^{(\infty)}$ in $L^{2}(I\times I)$, i.e.\@ $\norm{\bar{\K}^{(N)}-\K^{(\infty)}}_{L^{2}(I\times I)}\to 0$ as $N\to\infty$ according to Lemma~\ref{Lem:graph:inverse:convergence}. Consequently, Theorem~\ref{Thm:continuum:limit} implies that
 \begin{equation}\label{eq:proof:cont:limit:stoch:2}
  \sup_{t\in[0,T]}\norm*{\phi(\cdot,t)-\sum_{k=1}^{N}\phi_{k}^{(N)}\chi_{[(k-1)/N,k/N]}(\cdot)}_{L^{2}}\longrightarrow 0\qquad \text{as }N\to \infty.
 \end{equation}
 Furthermore, Proposition~\ref{Prop:approximation:averaged} ensures that
 \begin{equation}\label{eq:proof:cont:limit:stoch:3}
  \sup_{t\in[0,T]}\norm*{\phi^{(N)}(t)(\cdot,t)-u^{(N)}(t)}_{2} \longrightarrow 0\qquad \text{in probability as }N\to\infty.
 \end{equation}
 Combining~\cref{eq:proof:cont:limit:stoch:1,eq:proof:cont:limit:stoch:2,eq:proof:cont:limit:stoch:3} the claim follows.
\end{proof}

\section{Stability}

In this section we are going to study a more explicit model of oscillator networks. More precisely, we restrict here to graphons of the special form
\begin{equation}\label{eq:kernel:Kp}
 \K(x,y)=\K_{p}(x-y)
\end{equation}
with a non-negative and symmetric function $\K_{p}\colon I\to [0,\infty)$ which can be periodically extended to $\R$. The index $p$ indicates a parameter which affects the structure of the graphon (see Section~\ref{Sec:instability}). Note also that we drop here the label $\infty$ since we will only work with the continuum limit from now on.
\begin{remark}
 The assumption $\K(x,y)=\K_{p}(x-y)$ with $\K_{p}$ being symmetric directly gives that $\K_{p}$ in fact only depends on $\abs{x-y}$. This already suggests that the graphs approximated by such graphons have an underlying ordering as for example the one dimensional small-world model.
\end{remark}

Moreover, we also consider the special choices
\begin{equation}\label{eq:special:nonlinearity}
D(x)=K\sin(2\pi x),\qquad f\equiv 0 \qquad \text{and} \qquad \alpha>0
\end{equation}
with a constant $K\geq 0$, i.e.\@ we are interested in the equation
\begin{equation}\label{eq:instability:1}
 \ddot{\phi}+\alpha\dot{\phi}=K\int_{I}\K_{p}(y-x)\sin\bigl(2\pi(\phi(y)-\phi(x))\bigr)\dy.
\end{equation}
Obviously, each constant $\Phi\in\R$ is a steady state of~\eqref{eq:instability:1}, i.e.\@ $\Phi$ solves~\eqref{eq:instability:1}.

\begin{remark}
 Note that the calculations below are also true for $\alpha <0$ while the case $\alpha=0$ needs some adaptations. However, since for $\alpha\leq 0$ the constant stationary states are linearly unstable and we are mainly interested in stable states and their transition to instability, we restrict to $\alpha>0$ to keep the presentation simple.
\end{remark}

The main goal of this section will be to examine the linear stability of these constant steady states and to study how the stability properties change upon varying the structure of the graph. Our approach and the graph model we are considering is motivated mainly by~\cite{Med14b} and we also mainly follow the notation used there. 

\subsection{Stability of the constant steady state}\label{Sec:argument:stability}

We adopt and modify the strategy given in~\cite[Section~3]{Med14b}. In fact it is shown in~\cite[Lemma~3.5]{Med14b} that
\begin{equation*}
 \int_{I\times I}\K_{p}(y-x)\sin\bigl(2\pi(\phi(y)-\phi(x))\bigr)\dy\dx\equiv 0
\end{equation*}
due to the symmetry of $\K_{p}$. Thus, integrating~\eqref{eq:instability:1} over $I$ we obtain
\begin{equation}\label{eq:zero:moment}
 \frac{\dd^2}{\dt^2}\int_{I}\phi(x,t)\dx+\alpha \frac{\dd}{\dt}\int_{I}\phi(x,t)\dx=0.
\end{equation}
Note that the regularity of $\phi$ provided by Proposition~\ref{Prop:well:posed:cont:limit} allows to exchange the order of integration and differentiation. The latter equation can be solved explicitly and we precisely obtain
\begin{equation}\label{eq:zero:moment:explicit}
 \int_{I}\phi(x,t)\dx=\int_{I}\phi(x,0)\dx\ee^{-\alpha t}+\frac{1}{\alpha}\int_{I}\dot{\phi}(x,0)+\alpha \phi(x,0)\dx\bigl[1-\ee^{-\alpha t}\bigr].
\end{equation}
In particular, we have two conserved quantities of the evolution~\eqref{eq:zero:moment}, namely
\begin{equation*}
 \int_{I}\dot{\phi}(x,t)+\alpha \phi(x,t)\dx \quad \text{and}\quad -\frac{1}{\alpha}\int_{I}\dot{\phi}(x,t)\dx\ee^{\alpha t}.
\end{equation*}
In order to study the linear stability of a (constant) stationary state $\Phi$ of~\eqref{eq:instability:1}, we linearise around the latter and consider the spectrum of the resulting equation. More precisely, we plug the ansatz $\phi(x,t)=\Phi+\rho(x,t)$ into~\eqref{eq:instability:1} and keep only the linear expressions in the resulting equations which leads to
\begin{equation}\label{eq:linearised:0}
 \ddot{\rho}+\alpha\dot{\rho}=2\pi K\int_{I}\K_{p}(y-x)\bigl[\rho(y,t)-\rho(x,t)\bigr]\dy.
\end{equation}
Expanding $\rho$ as a Fourier series
\begin{equation*}
 \rho(x,t)=\sum_{k\in\Z}\hat{\rho}_{k}(t)\ee^{-2\pi \im k x} \quad \text{with}\quad \hat{\rho}_{k}=\int_{I}\rho(x,t)\ee^{2\pi \im k x}\dx
\end{equation*}
we obtain that~\eqref{eq:linearised:0} can be equivalently expressed by the following system of ordinary differential equation for the Fourier coefficients 
\begin{equation}\label{eq:ODE:Fourier}
 \frac{\dd^2}{\dt^2}\hat{\rho}_{m}+\alpha \frac{\dd^2}{\dt}\hat{\rho}_{m}=K\bigl[\widehat{\K}_{p}(m)-\widehat{\K}_{p}(0)\bigr]\hat{\rho}_{m}\quad \text{for } m\in\Z
\end{equation}
where $\widehat{\K}_{p}(m)=\int_{I}\K_{p}(x)\ee^{2\pi \im m x}\dx$.
\begin{remark}
 It will turn out that the Fourier coefficient corresponding to $m=0$ has to be treated separately since one of the eigenvalues corresponding to $m=0$ is zero. The same in principle also happens for the case of the Kuramoto model considered in~\cite{Med14b} but in this case the equation is only of first order in time and consequently the zeroth Fourier coefficient is just constant. In contrast, in our case this is no longer true but for $m=0$ we observe that~\eqref{eq:ODE:Fourier} is equivalent to~\eqref{eq:zero:moment} for which we have an explicit solution and in particular exponential convergence to a constant value. As a consequence, it suffices to consider only $m\neq 0$ in the following.
\end{remark}

In order to study the stability for $m\in\Z\setminus\{0\}$, we rewrite~\eqref{eq:ODE:Fourier} as a system of first order equations which yields
\begin{equation*}
 \begin{split}
  \frac{\dd}{\dt}\hat{\rho}_{m}&=\hat{\nu}_{m}\\
  \frac{\dd}{\dt}\hat{\nu}_{m}&=-\alpha \hat{\nu}_{m}+K\bigl[\widehat{\K}_{p}(m)-\widehat{\K}_{p}(0)\bigr]\hat{\rho}_{m}.
 \end{split}
\end{equation*}
The corresponding matrix then has the form
\begin{equation*}
 \begin{pmatrix}
  0 & 1\\
  K\bigl[\widehat{\K}_{p}(m)-\widehat{\K}_{p}(0)\bigr] & -\alpha
 \end{pmatrix}
\end{equation*}
and the eigenvalues can be computed to be
\begin{equation}\label{eq:eigenvalues}
 \lambda_{\pm}(m)=-\frac{\alpha}{2}\pm\sqrt{\Bigl(\frac{\alpha}{2}\Bigr)^{2}+K\bigl[\widehat{\K}_{p}(m)-\widehat{\K}_{p}(0)\bigr]}.
\end{equation}
In order to show stability we have to show that $\lambda_{\pm}(m)<0$ for all $m\in \Z\setminus\{0\}$. In fact, if the latter holds, $\hat{\rho}_{m}$ converges to zero as $t\to\infty$ while $\hat{\rho}_{0}(t)\to \hat{\nu}_{0}(0)/\alpha+\hat{\rho}_{0}(0)$ as $t\to\infty$. The last convergence is an immediate consequence of~\eqref{eq:zero:moment:explicit}. Consequently, Parseval's identity yields that $\rho(x,t)$ converges to this constant as well.

Thus, we are led to consider the expression $\widehat{\K}_{p}(m)-\widehat{\K}_{p}(0)$ more closely. In fact, we can rewrite this as
\begin{multline*}
 \widehat{\K}_{p}(m)-\widehat{\K}_{p}(0)=\int_{I}\K_{p}(x)\bigl[\ee^{2\pi \im m x}-1\bigr]\dx\\*
 =\int_{I}\K_{p}(x)\bigl[\cos(2\pi m x)-1\bigr]\dx+\im \int_{I}\K_{p}(x)\sin(2\pi m x)\dx.
\end{multline*}
To proceed, we recall that by assumption $\K_{p}(x)\sin(2\pi mx)$ is periodic with period one and antisymmetric. Consequently, we have $\int_{I}\K_{p}(x)\sin(2\pi m x)\dx=0$ which yields
\begin{equation}\label{eq:spectral:1}
 \widehat{\K}_{p}(m)-\widehat{\K}_{p}(0)=\int_{I}\K_{p}(x)\bigl[\cos(2\pi m x)-1\bigr]\dx.
\end{equation}
Since $\K_{p}\geq 0$ and $\cos(2\pi m \cdot)\leq 1$ we thus conclude $\widehat{\K}_{p}(m)-\widehat{\K}_{p}(0)\leq 0$. However, if $m\neq 0$ we obtain even more because $\K_{p}\not \equiv 0$ and thus $\widehat{\K}_{p}(m)-\widehat{\K}_{p}(0)<0$ for $m\neq 0$. Using this observation in~\eqref{eq:eigenvalues}, we obtain $\lambda_{\pm}(m)<0$ for all $m\in\Z\setminus\{0\}$ and thus linear stability of the steady state.

In summary, we have shown the following statement.

\begin{proposition}
\label{Prop:linear:stability}
 Each constant $\Phi\in\R$ is a stationary state to~\eqref{eq:instability:1}, which is linearly stable for $\alpha>0$.
\end{proposition}

\section{Approaching instability}\label{Sec:instability}

In this section, we will demonstrate that the constant steady states $\Phi$, which have been proven to be linearly stable in Section~\ref{Sec:argument:stability}, may only be marginally stable upon varying parameters in the model. More precisely, one immediately sees from~\eqref{eq:eigenvalues} that the spectrum of the linearised operator approaches the imaginary axis if we reduce the coupling strength $K$ between the oscillators. This means, that for a system with very weak coupling, although the steady state $\Phi$ is still linearly stable, a rather small perturbation suffices to destabilise the system. 

In addition to this, we are mainly interested in the question how the \emph{structure of the graph} influences the (linear) stability of steady states. The result that we obtain is similar to the previous observation, i.e.\@ for a variant of the small-world graph the system can be put arbitrarily close to instability by choosing the structure of the graph accordingly.

\subsection{A modification of the small-world network}

As a preparation of the following discussion of destabilisation, we precisely describe in this section the graph model we are going to consider and we collect several basic properties and auxiliary results. As already announced above, we want to look at a modification of the one dimensional small-world network and the corresponding graphon. 

The starting point for both models is the $m$-nearest-neighbour graph which consists of $N$ nodes such that $2m<N$ and which are arranged on a circle. This induces a canonical (discrete) distance between the nodes given by
\begin{equation*}
 \dist_{N}(k,\ell)=\min\bigl\{\abs{k-\ell}, N-\abs{k-\ell}\bigr\}.
\end{equation*}
This produces a graph $\tilde{G}_{N}=\langle [N], E(\tilde{G}_{N})\rangle$ with $E(\tilde{G}_{N})=\{(k,\ell)\in [N]\times [N]\;|\; 0<\dist_{N}(k,\ell)\leq m\}$. 

From this graph $\tilde{G}_{N}$ the classical Watts-Strogatz small-world graph $G_{N}$ is then constructed by rewiring several of these short range connections. More precisely, one iterates through $E(\tilde{G}_{N})$ and with probability $p$ one removes the current edge and rewires one end with a new random location. More details can be found for example in \cite{New03}.

Following~\cite{Med14b}, to take the limit $N\to \infty$ in this model, one fixes two parameters $r\in(0,1/2)$ and $p\in(0,1/2)$ and considers the sequence $G_{N}$ of small-world graphs constructed from the $\lfloor rN \rfloor$-nearest-neighbour graph $\tilde{G}_{N}$ with probability $p$. Then, in the limit $N\to\infty$, the sequence $G_{N}$ converges to a graphon $\SW_{p,r}$ which is given by
\begin{equation*}
 \SW_{p,r}=(1-p)W_{r}+p(1-W_{r})\quad \text{with}\quad W(x,y)=\begin{cases}
                                                       1 & \text{if }\dist(x,y)\leq r\\
                                                       0 & \text{else}.
                                                      \end{cases}
\end{equation*}
Here, $\dist(x,y)=\min\{\abs{x-y}, 1-\abs{x-y}\}$ is the natural generalisation of $\dist_{N}$. Introducing moreover
\begin{equation}\label{eq:definition:Gr}
 G_{s}(x)=\begin{cases}
           1 & \text{if } \pernorm(x)<s\\
           0 & \text{else}
          \end{cases}
\quad \text{with} \quad \pernorm(x)=\min\{\abs{x},1-\abs{x}\}\text{ for } x\in[-1,1]        
\end{equation}
the small-world model can also be expressed as
\begin{equation*}
 \SW_{p,r}(x,y)=pG_{1/2}(x-y)+(1-2p)G_{r}(x-y).
\end{equation*}

As already noted before, we want to consider in this work a slight modification of the small-world network. Precisely, instead of rewiring certain edges in $\tilde{G}_{N}$, we consider a model where we insert with probability $p$ new edges without removing any edge. For this model, we can proceed analogously to the small-world graph and we obtain in the limit $N\to\infty$ a graphon $\K_{p,r}(x,y)$. Precisely, writing by abuse of notation $\K_{p,r}(x,y)=\K_{p,r}(x-y)$ we have that
\begin{equation}\label{eq:modified:kernel}
 \K_{p,r}(x)=pG_{1/2}(x)+(1-2p)G_{r}(x) \quad \text{with }r\in(0,1/2)
\end{equation}
with $G_{s}$ as in~\eqref{eq:definition:Gr}.

To conclude this subsection let us prove for completeness that $G_{s}$ is a periodic function.

\begin{lemma}\label{Lem:periodicity:d}
 The functions $\pernorm(\cdot)$ and $G_{s}$ as given in~\eqref{eq:definition:Gr} can be extended periodically to all of $\R$ with period equal to one.
\end{lemma}

\begin{proof}
 We only show the claim for $\pernorm(\cdot)$ since the one for $G_{s}$ then follows immediately due to the definition. Since $\pernorm(\cdot)$ is only defined on $[-1,1]$ it suffices to verify that $\pernorm(x+1)=\pernorm(x)$ for all $x\in[-1,0]$. To see the latter, we note that $x\in[-1,0]$ implies
 \begin{equation*}
  \abs{x+1}=1+x=1-\abs{x} \quad \text{as well as} \quad 1-\abs{x+1}=1-1-x=-x=\abs{x}.
 \end{equation*}
 Together this shows
 \begin{equation*}
  \pernorm(x+1)=\min\{\abs{x+1},1-\abs{x+1}\}=\min\{\abs{x},1-\abs{x}\}=\pernorm(x)
 \end{equation*}
for $x\in[-1,0]$.
\end{proof}

\begin{remark}\label{Rem:Kp:even}
 Due to definition the function $\K_{p,r}$, or rather its periodic extension,  is symmetric, i.e.\@ $\K_{p,r}(-x)=\K_{p,r}(x)$ for all $x\in\R$
\end{remark}

\subsection{Moving the spectrum to the imaginary axis}

In this section, we will consider~\eqref{eq:instability:1} with the kernel $\K_{p,r}$ given by~\eqref{eq:modified:kernel} and we are going to show that analogously to the coupling strength $K$, in the limit $p,r\to 0$ the spectrum of the linearised operator comes arbitrarily close to the imaginary axis. For the type of graph under consideration this then shows that if local edges only connect vertices which are close together and if the density of long-range connections is small, the coupled system is close to being destabilised.

Thus, we want to show that the real part of the spectrum gets arbitrarily small if $p,r\to 0$. According to the discussion in Section~\ref{Sec:argument:stability}, it suffices to consider~\eqref{eq:spectral:1} for $m\neq 0$ with $\K_p$ replaced by $\K_{p,r}$, i.e.\@ the term
\begin{equation}\label{eq:unstable:1}
 \widehat{\K}_{p,r}(m)-\widehat{\K}_{p,r}(0)=\int_{I}\K_{p,r}(x)\bigl[\cos(2\pi m x)-1\bigr]\dx.
\end{equation}
We recall from~\cite[eq.\@~(4.1)]{Med14b} that
\begin{equation*}
 \int_{0}^{1}G_{s}(x)\cos(2\pi mx)\dx=\begin{cases}
                                       2s & \text{if } m=0\\
                                       \frac{1}{\pi m}\sin(2\pi m s) & \text{else}.
                                      \end{cases}
\end{equation*}
Using this in~\eqref{eq:unstable:1} together with~\eqref{eq:modified:kernel} yields
\begin{equation*}
 \widehat{\K}_{p,r}(m)-\widehat{\K}_{p,r}(0)=p\Bigl(\frac{\sin(\pi m)}{\pi m}-1\Bigr)+(1-p)\Bigl(\frac{\sin(2\pi m r)}{\pi m}-2r\Bigr).
\end{equation*}
One immediately checks that for $m\neq 0$ the expression on the right-hand side is strictly negative unless $r=p=0$ in which case we obtain zero. Thus, the spectrum of the linearisation in fact gets arbitrarily close to the imaginary axis if $p,r\to 0$.

\appendix

\section{Proof of Proposition~\ref{Prop:well:posed:cont:limit}}\label{Sec:proof:existence}

\begin{proof}[Proof of Proposition~\ref{Prop:well:posed:cont:limit}]
 The proof is standard and relies on rewriting~\eqref{eq:cont:limit:general} as a system of first order equations and applying the Banach fixed-point theorem. In fact, \eqref{eq:cont:limit:general} can be rewritten as
 \begin{equation*}
  \begin{split}
   \dot{\phi}&=\psi\\
   \dot{\psi}&=-\alpha \psi+ \int_{I}\K(x,y)D\bigl(\phi(y)-\phi(x)\bigr)\dy +f(\phi,t)
  \end{split}
 \end{equation*}
 Furthermore, exploiting the initial condition, we can integrate in time to obtain
  \begin{equation}\label{eq:int:cont:limit}
  \begin{split}
   \phi(x,t)&= g(x)+\int_{0}^{t}\psi(x,s)\ds\\
   \psi(x,t)&=h(x)-\int_{0}^{t}\biggl(\alpha \psi(x,s)- \int_{I}\K(x,y)D\bigl(\phi(y,s)-\phi(x,s)\bigr)\dy +f\bigl(\phi(x,s),s\bigr)\biggr)\ds.
  \end{split}
 \end{equation}
 We define now the operator $\B\colon \bigl(C([0,T],L^{\infty}(I))\bigr)^2\to \bigl(C([0,T],L^{\infty}(I))\bigr)^2$ via
   \begin{multline*}
   (\phi,\psi)\mapsto \Biggl(g(x)+\int_{0}^{t}\psi(x,s)\ds,\\*
   h(x)-\int_{0}^{t}\biggl(\alpha \psi(x,s)- \int_{I}\K(x,y)D\bigl(\phi(y,s)-\phi(x,s)\bigr)\dy +f\bigl(\phi(x,s),s\bigr)\biggr)\ds\Biggr).
  \end{multline*}
  Then, $(\phi,\psi)$ is a solution to~\eqref{eq:int:cont:limit} if and only if $(\phi,\psi)$ is a fixed-point for $\B$. The latter can be obtained easily by an application of the Banach fixed-point theorem provided $T$ is sufficiently small. We only show that $K$ is a contraction since the other requirements are obvious.
  \begin{multline*}
   \phantom{{}\leq{}}\norm*{\B(\phi^{(I)},\psi^{(I)})-\B(\phi^{(II)},\psi^{(II)})}_{(C([0,T],L^{\infty}))^2}\\*
   \shoveleft{\leq \sup_{t\in[0,T]}\sup_{x\in I}\Biggl[\abs*{\int_{0}^{t}\bigl(\psi^{(I)}-\psi^{(II)}\bigr)(x,s)\ds}+\biggl|\int_{0}^{t}\Bigl\{\alpha\bigl(\psi^{(I)}-\psi^{(II)}\bigr)(x,s)}\\*
   +\int_{I}\K(x,y)\bigl[D\bigl(\phi^{(I)}(y,s)-\phi^{(I)}(x,s)\bigr)-D\bigl(\phi^{(II)}(y,s)-\phi^{(II)}(x,s)\bigr)\bigr]\dy\\*
   \shoveright{+\bigl[f\bigl(\phi^{(I)}(x,s)\bigr)-f\bigl(\phi^{(II)}(x,s)\bigr)\bigr]\Bigr\}\ds\biggr|\Biggr]}\\*
   \shoveleft{\leq (1+\abs{\alpha})T\norm*{\psi^{(I)}-\psi^{(II)}}_{C([0,T],L^{\infty})}+L_{f}T\norm*{\phi^{(I)}-\phi^{(II)}}_{C([0,T],L^{\infty})}}\\*
   +\int_{0}^{t}\norm{\K}_{L^{\infty}(I^2)}L_{D}\int_{I}\abs*{\phi^{(I)}(y,s)-\phi^{(I)}(x,s)-\phi^{(II)}(y,s)+\phi^{(II)}(x,s)}\dy\ds.
  \end{multline*}
  Using $\int_{I}\abs*{\phi^{(I)}(y,s)-\phi^{(I)}(x,s)-\phi^{(II)}(y,s)+\phi^{(II)}(x,s)}\dy\leq 2\norm{\phi^{(I)}-\phi^{(II)}}_{C([0,T],L^{\infty})}$ we further obtain
  \begin{multline}
    \phantom{{}\leq{}}\norm*{\B(\phi^{(I)},\psi^{(I)})-\B(\phi^{(II)},\psi^{(II)})}_{(C([0,T],L^{\infty}))^2}\\*
    \shoveleft{\leq (1+\abs{\alpha})T\norm*{\psi^{(I)}-\psi^{(II)}}_{C([0,T],L^{\infty})}}+\Bigl(L_{f}+2L_{D}\norm{\K}_{L^{\infty}(I^2)}\Bigr)T\norm*{\phi^{(I)}-\phi^{(II)}}_{C([0,T],L^{\infty})}\\*
    \leq \Bigl(L_{f}+2L_{D}\norm{\K}_{L^{\infty}(I^2)}+(1+\abs{\alpha})\Bigr)T\norm*{(\phi^{(I)},\psi^{(I)})-(\phi^{(II)},\psi^{(II)})}_{(C([0,T],L^{\infty}))^2}.
   \end{multline}
 Thus, choosing $T$ sufficiently small (while smallness depends here only on $L_{f}$, $L_D$, $\K$ and $\alpha$) the operator $\B$ is a contraction and thus yields the existence of a unique fixed-point in $(C([0,T],L^{\infty}))^2$. The proof is finished by extending this solution in the usual way while regularity in time follows directly from the structure of $\B$.
\end{proof}

\textbf{Acknowledgments:} This work was supported by a Lichtenberg professorship of the VolkswagenStiftung.

\end{document}